\documentclass[11pt, oneside]{article}   	
\usepackage{amsmath, amsthm, amsfonts, amssymb}
\usepackage{newtxtext,newtxmath}
\usepackage[]{hyperref}
\newcommand\Ex{{\mathbb E}}
\newcommand\W{{\mathbb W}}

\newcommand\Prob{{\mathbb P}}

\newcommand\X{{\mathbb X}}
\newcommand\Y{{\mathbb Y}}

\newcommand\cC{{\mathcal C}}

\newcommand\cP{{\mathcal P}}

\newcommand\N{{\mathbb N}}

\newcommand\R{{\mathbb R}}

\newcommand\C{{\mathbb C}}

\newcommand\x{{\mathbf x}}
\newcommand\y{{\mathbf y}}

\newcommand\norm[1]{\|#1\|}

\newcommand\bra[1]{\langle #1 \rangle}

\newtheorem{theorem}{Theorem}[section]

\newtheorem{lemma}[theorem]{Lemma}

\theoremstyle{definition}
\newtheorem{definition}[theorem]{Definition}

\theoremstyle{remark}
\newtheorem{remark}[theorem]{Remark}

\title{Beta Laguerre processes in a high temperature regime}

\author{Hoang Dung Trinh\footnote{Faculty of Mathematics Mechanics Informatics, University of Science, Vietnam National University, Hanoi, Vietnam.
\newline Email: thdung.hus@gmail.com} 
\and
Khanh Duy Trinh\footnote{Global Center for Science and Engineering, Waseda University, Japan.
\newline
Email: trinh@aoni.waseda.jp 
} 
}

\begin{document}
\maketitle

\begin{abstract}

Beta Laguerre processes which are generalizations of the eigenvalue process of Wishart/Laguerre processes can be defined as the squares of radial Dunkl processes of type B. In this paper, we study the limiting behavior of their empirical measure processes.
By the moment method, we show the convergence to a limit  in a high temperature regime, a regime where $\beta N \to const \in (0, \infty)$, where $\beta$ is the inverse temperature parameter and $N$ is the system size. This is a dynamic version of a recent result on the convergence of the empirical measures of beta Laguerre ensembles in the same regime.

\medskip

		
\medskip
	
	
%
%
\end{abstract}

\section{Introduction}
The so-called beta Laguerre processes (of Ornstein--Uhlenbeck type) solve the following system of stochastic differential equations (SDEs)
\begin{equation}
\begin{cases}
	d\lambda_i(t)=\sqrt{2\lambda_i}  d  b_i(t) - \lambda_i(t)  d t  + \alpha  d  t + \dfrac{\beta}2 \sum\limits_{j : j \neq i} \dfrac{2\lambda_i(t)}{\lambda_i(t) - \lambda_j(t)} dt,\\
	\lambda_i(0) = \lambda_0^{(N,i)},
\end{cases}
 i = 1, \dots, N,
\label{SDE-lambda-intro}
\end{equation}
where $\{b_i(t)\}_{i = 1, \dots, N}$ are independent standard Brownian motions, $\alpha, \beta > 0$ are parameters, and $0 \le \lambda_0^{(N,1)}  \le \cdots \le \lambda_0^{(N,N)}$ are initial data. Without the drift term $-\lambda_i dt$, SDEs of the above form  generalize the eigenvalue process of Wishart processes $(\beta = 1)$ \cite{Bru-1989, Bru-1991} and Laguerre processes $(\beta = 2)$ \cite{Demni-2007, Katori-Tanemura-2004, Konig-OConnell-2001}, and hence the name. When $\beta \ge 1$ and $\alpha > 0$, the system of SDEs~\eqref{SDE-lambda-intro} has a unique strong solution with no collisions (cf.\ \cite[\S 6.4]{Graczyk-Jacek-2014}). There has not been any direct approach to study the case $\beta \in (0, 1)$ yet. However, as observed in \cite{Demni-2007-arxiv}, beta Laguerre processes can be defined as  the squares of radial Dunkl processes of type B, extending the range of parameters in their definition to $\beta > 0$ and $\alpha > 1/2$.

Starting from any initial condition $\{\lambda_0^{(N,i)}\}_{1\le i \le N}$, as $t \to \infty$, the joint distribution of  $\{\lambda_i (t)\}_{1 \le i \le N}$ converges weakly to a distribution with joint density 
\begin{equation}\label{BLE}
	\frac{1}{Z_{N, \alpha, \beta}}\prod_{i < j} |\lambda_j - \lambda_i|^\beta \prod_{l = 1}^N \lambda_l^{\alpha - 1} e^{- \lambda_l}, \quad (0 \le \lambda_1 \le  \dots \le \lambda_N),
\end{equation}
with $Z_{N, \alpha, \beta}$ the normalizing constant, which belongs to a family of beta Laguerre ensembles ($\beta$LEs for short) in random matrix theory. That fact follows from  the explicit joint density of radial Dunkl processes of type B (see Sect.\ 2).  The ensembles which generalize the eigenvalue distribution of Wishart matrices $(\beta = 1)$ and Laguerre matrices $(\beta = 2)$  can be realized as the eigenvalues of a random tridiagonal matrix~\cite{DE02}. Furthermore, by expressing the joint density in the form,
\[
	const \times \exp \bigg(- \beta \Big( \frac{1}{2} \sum_{i \neq j} -  \log |\lambda_j - \lambda_i| +  \sum_{l = 1}^N  V_N(\lambda_l) \Big) \bigg),
\]
$\beta$LEs are viewed as the equilibrium measure of a Coulomb log-gas on $[0, \infty)$ under a potential $V_N$ at the inverse temperature $\beta$.

For $\beta$LEs, the limiting behavior of eigenvalues has been well studied.
For fixed $\beta$, the empirical distribution of the eigenvalues under a suitable scaling  converges weakly to the Marchenko--Pastur distribution as $N\to \infty$, almost surely \cite{DE06}, which is a natural extension of the well-known Marchenko--Pastur law for Wishart and Laguerre matrices. Readers who are interested in Wishart and Laguerre matrices are referred to a monograph \cite{Pastur-book}. We note here that for such limiting behavior in case $\beta$ is fixed, the parameter $\alpha$ varies as a function of $N$ (and of $\beta$ as well), determining the parameter of the Marchenko--Pastur distribution. The Marchenko--Pastur law still holds when $\beta = \beta(N)$ varies as long as $\beta N \to \infty$ \cite{Trinh-Trinh-2021}. However, in the so-called high temperature regime, $\beta N \to 2c \in (0, \infty)$,  with probability one, the empirical distribution of the eigenvalues~\eqref{BLE}, without scaling and with fixed $\alpha$, converges weakly to a limiting probability measure $\nu_{\alpha, c}$ related to associated Laguerre polynomials \cite{Trinh-Trinh-2021} (see also \cite{Allez-Wishart-2013}).

The aim of this paper is to study the limiting behavior of the empirical measure process
\begin{equation}\label{empirical}
	\mu^{(N)}_t = \frac1N \sum_{i = 1}^N \delta_{\lambda_i(t)} 
\end{equation}
of the beta Laguerre process~\eqref{SDE-lambda-intro}
in a high temperature regime where $\beta N \to 2c \in (0, \infty)$, with $c>0$ and $\alpha > 1/2$ being fixed. Here $\delta_\lambda$, for $\lambda \in \R$, denotes the Dirac measure. Note that the case where $\beta$ is fixed was studied in \cite{Duvillard-Guionnet-2001} in which a dynamic version of the Marchenko--Pastur law was established. A method to deal with this kind of problems has been well developed \cite{Duvillard-Guionnet-2001, Cepa-Lepingle-1997, Rogers-Shi-1993}. By imitating arguments from those works, we can immediately derive the following result.

\begin{theorem}\label{thm:intro}
Assume that the initial measure $\mu_0^{(N)}$ converges weakly to a probability measure $\mu_0$ and satisfies  
\begin{equation}\label{log}
	\sup_N \int \log(1+x) d\mu_0^{(N)} < \infty.
\end{equation}
Then for any $T>0$, the sequence $(\mu_t^{(N)})_{0 \le t \le T}$ is tight in the space $\cC([0, T], \cP(\R_{\ge 0}))$ and any subsequential limit is supported on the set of continuous probability measure-valued processes $(\mu_t)_{0 \le t \le T}$ satisfying the integro-differential equation
\begin{align}
	\bra{\mu_t, f} &=  \bra{\mu_0, f} + \int_0^t \bra{\mu_s, \alpha f' - xf' + x f''} ds \notag\\
	&\quad + c \int_0^t \left( \iint \frac{xf'(x) - yf'(y)} {x - y} d\mu_s(x) d\mu_s(y)  \right) ds, \quad t \in [0, T], \label{equation-mu-intro}
\end{align}
for all $f \in C_b^2 = \{f \colon [0, \infty) \to \R : f, f', f'' \text{ bounded}\}$ with $xf', xf''$ bounded.
Here $\cC([0, T], \cP(\R_{\ge 0}))$ is the space of continuous mappings from $[0,T]$ to the space $\cP(\R_{\ge 0})$ of probability measures on $\R_{\ge 0} = [0, \infty)$ endowed with the uniform topology, and $\bra{\mu, f} = \int f d\mu$ for a measure $\mu$ and an integrable function $f$.
\end{theorem}

By this theorem, the sequence $(\mu_t^{(N)})_{0 \le t \le T}$ will converge in distribution to a deterministic limit once the integro-differential equation~\eqref{equation-mu-intro} is shown to have a unique solution. This paper is not devoted to study the integro-differential equation in more details. Instead, we are going to use the moment method  to establish the convergence of $(\mu_t^{(N)})_{0 \le t \le T}$. By the moment method, we simply mean that the limiting behavior of the empirical measure processes can be derived by studying their moment processes. Under some moments assumptions (\textbf{H1} and \textbf{H2} in Sect.~3), we will show by induction that the $k$th moment process of $\mu_t^{(N)}$ converges in probability (as random elements in the space $\cC([0, T], \R)$ of continuous functions on $[0,T]$ endowed with the uniform norm) to a deterministic limit $m_k(t)$. 
Let $\mu_t$ be the unique probability measure-valued process with moments $\{m_k(t)\}$. (It is unique under our moments assumptions.) Then the convergence of every moment process implies that the sequence $(\mu_t^{(N)})_{1 \le t \le T}$ converges in probability to $(\mu_t)_{0 \le t \le T}$ as $N \to \infty$ (as random elements in $\cC([0, T], \cP(\R_{\ge 0}))$). 
Our main results can be summarized in the following diagram 
\begin{equation}\label{diagram}
		\begin{matrix}
			\mu_t^{(N)} &{\xrightarrow{ {N \to \infty} }}_{(i)} &\mu_t~~~~ \\
			~~~~~~{}_{(iii)}\Big\downarrow t \to \infty	&	&{}_{(ii)}\Big\downarrow t \to \infty	\\
			\beta LE(N)	& {\xrightarrow{{N \to \infty}}}_{(iv)}	& \nu_{\alpha, c}~~~
		\end{matrix}.
\end{equation}
Here (i) and (ii) are the main results in this paper with (i) stated more precisely in Theorem~\ref{thm:moment} and (ii) proved in Subsect.~\ref{subsec:bLE}. (iii) and (iv) were already mentioned above.

We have not  been aware of the use of the moment method in studying empirical measure processes yet. Thus, a general result on the method is given in Appendix~\ref{appendix-moment}. We note here that the moment method also works for the following models: Dyson's Brownian motion models which were already studied in \cite{Cepa-Lepingle-1997, Rogers-Shi-1993}, beta Laguerre processes (the usual type without the drift term $-\lambda_i dt$) and beta Laguerre processes in a regime where $\beta N \to \infty$.

The paper is organized as follows. In Sect.~2, we shortly introduce the type-B radial Dunkl process of Ornstein--Uhlenbeck type, and then define beta Laguerre processes. The limiting behavior of the empirical measure processes is studied in Sect.~3.

\section{Beta Laguerre processes}

\subsection{The B-type radial Dunkl process of Ornstein--Uhlenbeck type} 

Consider the closed subset of $\R^N$ given by
\[
{\W}_B:=\{{\x}\in {\R}^N : 0\leq x_1\leq \cdots \leq x_N \}.
\]
The B-type radial Dunkl process of Ornstein--Uhlenbeck type is defined as the Markov process with infinitesimal generator
\[
L_k [f](\x):=\frac{1}{2}\sum_{i=1}^{N}\frac{\partial^2}{\partial x_i^2}f(\x)+\sum_{i=1}^{N}\Big\{\frac{k_1}{x_i}+k_2\sum_{j:j\neq i}\frac{2 x_i}{x_i^2-x_j^2}-\frac{x_i}{2}\Big\}\frac{\partial}{\partial x_i}f(\x)
\]
for suitable $f\in C^2(\W_B)$. The two parameters $k_1, k_2>0$ are the \emph{multiplicities} of the root system of type B, which is expressed in terms of the canonical basis vectors $\{e_i\}_{i=1}^N$ as
\[
B_N:=\{e_i-e_j, 1\leq i\neq j\leq N\}\cup\{\pm(e_i+e_j), 1\leq i<j\leq N\}\cup\{\pm e\}_{i=1}^N.
\]

 The transition density of the Markov process was found in \cite{Rosler-Voit-1998}. Let $\hat{p}(t,\y|\x)$ be the transition density (the density of arriving at $\y$ after a time $t>0$ having started from $\x$) of the process without confinement (that is, without the restoring drift term $-x_i/2$). Then the transition density of the Ornstein--Uhlenbeck type process is given by $p(t,\y|\x)=\hat{p}(1-e^{-t},\y|\x e^{-t/2})$ (Sect.~10 in \cite{Rosler-Voit-1998}) 
\begin{align}
p(t,\y|\x)=&\frac{1}{c_k(1-e^{-t})^{N/2}}\prod_{i=1}^{N}\frac{y_i^{2k_1}}{(1-e^{-t})^{k_1}}\prod_{1\leq m<n\leq N}\Big(\frac{y_n^2-y_m^2}{1-e^{-t}}\Big)^{2 k_2}\notag\\
&\times \exp\Big(-\frac{\|\y\|^2+\|\x\|^2e^{-t}}{2(1-e^{-t})}\Big)\sum_{\sigma\in W_B}E_k\Big(\frac{\x e^{-t/2}}{\sqrt{1-e^{-t}}},\frac{\sigma\y}{\sqrt{1-e^{-t}}}\Big). \label{joint-density-of-X}
\end{align}
Let us now explain notations in the above formula. The reflection operators along the root system generate the Weyl group $W_B$ of all permutations and component-wise sign changes of vectors in $\R^N$. The function $E_k$ is the \emph{Dunkl kernel}, the joint eigenfunction of Dunkl operators \cite{Dunkl-1989, Dunkl-1991} of type B, and the explicit form of the sum over $\sigma\in W_B$ is given by a multivariate hypergeometric function \cite{Baker-Forrester-1997}, though we do not require its explicit form here. Finally, the normalization constant $c_k$ is given by the Selberg integral
\[
c_k:=2^N N!\int_{\W_B}e^{-\|\x\|^2/2}\prod_{l=1}^{N}x_l^{2k_1}\prod_{1\leq i<j\leq N}(x_j^2-x_i^2)^{2 k_2} d^N\x.
\]
We have used $\norm{\cdot}$ to denote the Euclidean norm in $\R^N$.

The process can be expressed in SDE form by reading off the infinitesimal generator: if we denote the process by $X(t)$ with $X(0)=\x$, then each component of its SDEs reads
\begin{equation}\label{radial-type-B}
 d X_i(t)= d b_i(t)+\Big(\frac{k_1}{X_i(t)}+k_2\sum_{j:j\neq i}\frac{2X_i(t)}{X_i^2(t)-X_j^2(t)}-\frac{X_i(t)}{2}\Big) d t, \quad i = 1, \dots, N,
\end{equation}
with $\{b_i(t)\}_{i = 1, \dots, N}$ standard Brownian motions. The above SDEs can also be treated via an approach in \cite{Cepa-Lepingle-1997} (see also \cite{Demni-2007-arxiv}).

Because the law $p(t,\y|\x)$ is controlled by Gaussian functions, we can use an inequality (\cite{Rosler-1999})
\begin{equation}\label{Dunkl-kernel-bound}
\sum_{\sigma\in W_B}E_k(\x,\sigma\y)\leq 2^N N! \exp(\|\x\|\|\y\|),
\end{equation}
to show that $
	\Ex[\|X_t\|^{2m}] 
$
is uniformly bounded in $t$, for each $m \in \{1,2,\dots\}$. This is a crucial property we need when using the moment method.

\subsection{Beta Laguerre processes}
Let $\lambda_i = X_i^2/2, i = 1, \dots, N$, with $\{X_i\}$ the solution of the SDEs \eqref{radial-type-B}. Then $\{\lambda_i\}$, called beta Laguerre processes, satisfy the following SDEs
\[
 d \lambda_i= \sqrt{2\lambda_i}  d  b_i - \lambda_i  d t  + \Big(k_1 + \frac12\Big) d  t + k_2 \sum_{j : j \neq i} \frac{2\lambda_i}{\lambda_i - \lambda_j} dt, \quad i = 1, \dots, N,
\]
which are exactly the SDEs~\eqref{SDE-lambda-intro} with $\alpha = k_1 + 1/2 > 1/2$ and $\beta = 2 k_2 > 0$. For $\beta \in \{1,2\}$, they are realized as the eigenvalue process of Wishart/Laguerre processes \cite{Bru-1989, Bru-1991, Katori-Tanemura-2004, Konig-OConnell-2001}. Recall that when $\beta \ge 1$ and $\alpha > 0$, the above SDEs are defined in the usual sense and $\{\lambda_i\}$ never collide  (cf. \cite{Graczyk-Jacek-2014}). 

It is clear from the explicit expression for the joint density of $\{X_i(t)\}_{1 \le i \le N}$ in \eqref{joint-density-of-X} that the distribution of $\{\lambda_i(t)\}_{i \le 1 \le N}$, starting from any initial point, converges weakly to the beta Laguerre ensemble~\eqref{BLE} as $t \to \infty$.

\section{Convergence of the empirical measure process}

\subsection{Assumptions}
We now study the limiting behavior of the empirical measure process $\mu^{(N)}_t$ defined in equation~\eqref{empirical}
in the regime where $\beta N \to 2 c \in (0, \infty)$. For simplicity, let $c \in (0, \infty)$ be fixed and $\beta = 2c / N$ in the SDEs \eqref{SDE-lambda-intro}. 
We make the following assumptions on initial data.

\textbf{H1.} Each moment of $\mu_0^{(N)}$ converges, that is, for each $k=1,2,\dots$, there exists a number $a_k$ such that
\[
	\bra{\mu_0^{(N)}, x^k} = \frac1N \sum_{i = 1}^N (\lambda_0^{(N, i)})^k \to a_k\quad \text{as} \quad N \to \infty .
\]

\textbf{H2.}  
The sequence of initial moments $\{a_k\}$ does not grow too fast in the sense that
\begin{equation}\label{upper-lambdak}
	\sum_{k = 1}^\infty \Lambda_k^{-\frac1{2k}} = \infty,
\end{equation}
where $\{\Lambda_k\}$ is defined recursively as 
\[
		\Lambda_1 = (\alpha + c) \vee a_1, \quad \Lambda_k = (\alpha + k - 1 + ck) \Lambda_{k - 1} \vee a_k, \quad k = 2, 3, \dots.
\]
As we will see in Lemma~\ref{lem:limit-moment}, the number $\Lambda_k$ defined in that way gives an upper bound for the $k$th limiting moment process. Then the condition~\eqref{upper-lambdak} is assumed based on Carleman's sufficient condition to ensure that moments uniquely determine the probability measure.

Note that Conditions \textbf{H1} and \textbf{H2} together imply the conditions in the assumption of Theorem~\ref{thm:intro}. Indeed, under Condition \textbf{H2}, the sequence of moments $\{a_k\}$ satisfies
\[
		\sum_{k = 1}^\infty a_k^{-\frac1{2k}}  \ge 	\sum_{k = 1}^\infty \Lambda_k^{-\frac1{2k}} = \infty.
\] 
This is Carleman's sufficient condition under which a probability measure $\mu_0$ on $[0, \infty)$ whose moments match the sequence $\{a_k\}$ is unique. Together with Condition \textbf{H1}, it follows that the sequence of probability measures $\mu_0^{(N)}$ converges weakly to $\mu_0$  (see \cite[Theorem 30.2]{Billingsley} or \cite[\S 3.3.5]{Durrett-book}, for example). Since $\log(1+x) \le x$, for $x \ge 0$, the condition~\eqref{log} is clear.

\subsection{A standard method}

Let us mention a key idea in the proof of Theorem~\ref{thm:intro}.
The proof relies on the following formula which is a direct application of It\^o's formula 
\begin{align}
	d \bra{\mu^{(N)}_t, f(x)} &= \frac 1N \sum_{i = 1}^ N d f (\lambda_i) = \frac1N \sum_{i = 1}^N \Big( f'(\lambda_i) d\lambda_i + \frac12 f''(\lambda_i) (2 \lambda_i)dt  \Big) \notag\\
	&=  \sum_{i = 1}^N \frac1N  \sqrt{2\lambda_i} f'(\lambda_i) db_i +  \frac1N \sum_{i = 1}^N  f'(\lambda_i) (-\lambda_i + \alpha) dt  \notag\\
	&\quad + \frac 1N \sum_{i = 1}^N f'(\lambda_i) \frac{c}{N}\sum_{j \neq i} \frac{2 \lambda_i}{\lambda_i - \lambda_j}dt  + \frac1N \sum_{i = 1}^N \lambda_i f''(\lambda_i) dt \notag\\
	&= \sum_{i = 1}^N \frac1N  \sqrt{2\lambda_i} f'(\lambda_i) db_i + \bra{\mu^{(N)}_t, \alpha f'(x) - x f'(x) + x f''(x)}dt \notag\\
	&\quad +  c \iint \frac{xf'(x) - yf'(y)} {x - y} d\mu^{(N)}_t(x) d\mu^{(N)}_t(y)dt  \notag\\
	&\quad  - \frac cN \bra{\mu^{(N)}_t, xf''(x) + f'(x)} dt, \label{Ito-for-f}
\end{align}
for $f \in C^2(\R_{\ge 0})$. Here we have used the symmetry to deduce the last two terms  
\begin{align*}
	&\frac 1N \sum_{i = 1}^N f'(\lambda_i) \frac{c}{N}\sum_{j \neq i} \frac{2 \lambda_i}{\lambda_i - \lambda_j}dt = \frac c{N^2} \sum_{i \neq j} \frac{\lambda_i f'(\lambda_i) - \lambda_j f'(\lambda_j)}{\lambda_i - \lambda_j} dt\\
	&\quad =  \frac c{N^2} \sum_{i , j} \frac{\lambda_i f'(\lambda_i) - \lambda_j f'(\lambda_j)}{\lambda_i - \lambda_j}dt - \frac{c}{N^2} \sum_{i=j} (\lambda_i f''(\lambda_i) + f'(\lambda_i)) dt\\
	& \quad = c \iint \frac{xf'(x) - yf'(y)} {x - y} d\mu^{(N)}_t(x) d\mu^{(N)}_t(y)dt - \frac cN \bra{\mu^{(N)}_t, xf''(x) + f'(x)} dt.
\end{align*}
Note that singular terms $\frac{2\lambda_i}{\lambda_i - \lambda_j}dt$ as in the system of SDEs~\eqref{SDE-lambda-intro} have been removed in~\eqref{Ito-for-f}.
Then the arguments can run in exactly the same way as those used  in \cite{Cepa-Lepingle-1997, Rogers-Shi-1993}, and hence we omit the details here.

\begin{remark}
Assume that $\mu_t$ is a probability measure-valued process satisfying the equation \eqref{equation-mu-intro}. Let 
\[
	S = S(t,z) = \bra{\mu_t, (\cdot - z)^{-1}} =  \int \frac{d\mu_t(x)}{x - z}, \quad (t \ge 0, z \in \C \setminus \R),
\]
be the Stieltjes transform of $\mu_t$. Then the equation \eqref{equation-mu-intro} with $f = 1/(x-z)$ yields the following partial differential equation for $S$, 
\begin{align*}
	\frac{\partial S}{\partial t} = S + (2 + z - \alpha) \frac{\partial S}{\partial z} + z \frac{\partial^2 S}{\partial z^2} + c \left (S^2 + 2 z S \frac{\partial S}{\partial z} \right).
\end{align*}
If the above equation admits a unique solution, then so does the equation \eqref{equation-mu-intro}. At present, we do not know how to deal with these equations.
\end{remark}

\subsection{The moment method}
In this section, we introduce the moment method to study the limiting behavior of $\mu_t^{(N)}$. We first show the following result.
\begin{theorem}\label{thm:moment}
	Assume that Condition \textbf{H1} is satisfied. Then for any $k = 1,2,\dots,$ the $k$th moment process 
\[
	S_k^{(N)}(t) = \frac1N \sum_{i = 1}^N \lambda_i(t)^k
\]	
converges in probability to a deterministic differentiable function $m_k(t)$ which is defined inductively as the solution to the following initial value ODE 
\begin{equation}\label{ODE-moment}
\begin{cases}
	m_k'(t) = -  k m_k(t) + k\left( (\alpha + k - 1) m_{k - 1}(t) + c \sum_{i = 0}^{k - 1} m_i(t) m_{k - i - 1}(t) \right),\\
	m_k(0) = a_{k},
\end{cases}
\end{equation}
where $m_0 \equiv 1$. To be more precise, this means that for any $T > 0$, as random elements in the space of continuous functions $\cC([0, T], \R)$ endowed with the uniform norm, the sequence $\{S_k^{(N)}(t)\}$ converges in probability to $m_k(t)$.
\end{theorem}

We need some preparations to prove this theorem. To begin with, let us express equation~\eqref{Ito-for-f} with $f = x^k$ in the following form
\begin{align*}
	d S_k^{(N)}(t)  &= \sum_{i = 1}^N \frac{k}{N} \sqrt{2\lambda_i} \lambda_i^{k-1} db_i - k S_k^{(N)}(t) dt \\
	&\quad + \alpha k S_{k-1}^{(N)}(t) dt + c k \sum_{i = 0}^{k-1} S_i^{(N)}(t) S_{k-i-1}^{(N)}(t) dt  \\
	&\quad + k(k-1) S_{k-1}^{(N)}(t) dt - \frac{c k^2}N S_{k-1}^{(N)}(t) dt \\
	&=: dM_k^{(N)}(t) - k  S_k^{(N)}(t)dt +  F_k^{(N)}(t) dt.
\end{align*}
Here 
\[
	M_k^{(N)}(t) =\frac{k}{N} \sum_{i = 1}^N \int_0^t   \sqrt{2\lambda_i(s)} \lambda_i(s)^{k-1} db_i(s) = \frac{\sqrt 2 k}{N} \sum_{i = 1}^N \int_0^t    \lambda_i(s)^{k-1/2} db_i(s) 
\] is a martingale, because of the uniform boundedness of $
	\Ex[\|X_t\|^{2m}] 
$ (a statement following the equation~\eqref{Dunkl-kernel-bound}), with the quadratic variation 
\begin{equation}\label{quadratic-of-M}
	\bra{M_k^{(N)}}_t = \frac{2k^2}{N}   \int_0^t \frac{\sum_{i = 1}^N \lambda_i(s)^{2k-1}}{N}ds,
\end{equation}
and 
\begin{equation}\label{FkN}
	F_k^{(N)}(t) = \left(  k(\alpha + k-1) - \frac{c k^2}N \right) S_{k-1}^{(N)}(t) + c k \sum_{i = 0}^{k-1} S_i^{(N)}(t) S_{k-i-1}^{(N)}(t) .
\end{equation}
Now we write $S_k^{(N)}(t)$ in the integral form
\begin{align}
	S_k^{(N)}(t) &= \bra{\mu^{(N)}_0, x^k}  - k \int_0^t S_k^{(N)}(s)ds + M_k^{(N)}(t) +  \int_0^t F_k^{(N)}(s) ds	\notag \\
	&=: \bra{\mu^{(N)}_0, x^k}  - k \int_0^t S_k^{(N)}(s)ds + \Phi_k^{(N)}(t).\label{integral-form}
\end{align}
Here note that $\Phi_k^{(N)}(t)$ is a continuous function with $\Phi_k^{(N)}(0)=0$. In addition, observe that the above is an ODE for $\Psi(t) =  \int_0^t S_k^{(N)}(s)ds$. Thus $S_k^{(N)}(t)$ has the following explicit expression
\begin{equation}\label{explicit-solution}
	S_k^{(N)}(t)  =  \bra{\mu^{(N)}_0, x^k} e^{-kt} + \Phi_k^{(N)}(t) - k \left(\int_0^t \Phi_k^{(N)}(s) e^{ks} ds \right) e^{-kt}.
\end{equation}

Let $T$ be fixed. Let
$\X = (\cC([0, T], \R),\norm{\cdot})$ be the space of continuous functions on $[0,T]$ endowed with the supremum norm. Then $\X$ is a complete separable metric space. We consider $S_k^{(N)}, M_k^{(N)}$ and $F_k^{(N)}$ as random elements on $\X$.
\begin{definition}
Let $X^{(N)}$ and $X$ be $\X$-valued random elements defined on the same probability space. The sequence $X^{(N)}$ is said to converge in probability to $X$ if $\norm{X^{(N)} - X}$ converges in probability to $0$, that is, for any $\varepsilon > 0$,
\[
	\lim_{N \to \infty} \Prob(\norm{X^{(N)} - X} \ge \varepsilon) = 0.
\]
Note that when $X$ is deterministic, then the condition that  $X^{(N)}$ is defined on the same probability space is not necessary.
\end{definition}

The addition and multiplication operators on  $\X$ are defined pointwisely as usual. Based on the estimates that 
\[
	\norm{x + y} \le \norm{x} + \norm{y}, \quad \norm{xy} \le \norm{x} \norm{y},\quad x, y \in \X,
\]
we can easily show the following.
Assume that $X^{(N)}$ (resp.~$Y^{(N)}$) converges to $X$ (resp.~$Y$) in probability (as random elements on $\X$). Then the following hold.
\begin{itemize}
\item[\rm(i)] $X^{(N)} + Y^{(N)}$ (resp.~$X^{(N)}Y^{(N)}$) converges to $X+Y$  (resp.~$XY$) in probability.

\item[\rm(ii)] $\int_0^t X^{(N)}(s) ds$ converges to $\int_0^t X(s) ds$ in probability.

\end{itemize}

Back to our problem, we now show that the martingale part $M_k^{(N)}$ converges in probability to zero.

\begin{lemma}\label{lem:MTG-part}
	$M_k^{(N)}$ converges in probability to $0$ (in $\X$) as $N \to \infty$.
\end{lemma}
\begin{proof}
By using Doob's martingale inequality, we first estimate  
	\[
		\Prob \left(\norm{ M_k^{(N)}} \ge \varepsilon	\right) = \Prob \left(\sup_{0 \le t \le T} |M_k^{(N)}(t)| \ge \varepsilon	\right) \le \frac{\Ex[M_k^{(N)}(T)^2]}{\varepsilon^2} .
	\]
From this, it suffices to show that $\Ex[M_k^{(N)}(T)^2] \to 0$ as $N \to \infty$. Note from the quadratic formula~\eqref{quadratic-of-M} that 
\[
	\Ex[M_k^{(N)}(T)^2]  = \frac{2k^2}{N}   \Ex\bigg[ \int_0^T \frac{\sum_{i = 1}^N \lambda_i(s)^{2k-1}}{N}ds \bigg] = \frac{2k^2}{N}  \int_0^T \Ex[S_{2k-1}^{(N)}(s)] ds.
\]
Therefore, it now suffices to show that for each fixed $k$, there is a constant $D_k$ such that
\[
	s_k^{(N)}(t) := \Ex[S_k^{(N)}(t)] \le D_k, 
\]
for all $t \in [0, T]$ and all $N$.

Take the expectation in both sides of the identity \eqref{integral-form}, we get that 
\[
	s_k^{(N)}(t) =  \bra{\mu^{(N)}_0, x^k}  - k \int_0^t s_k^{(N)}(s) ds + \int_0^t \Ex[F_k^{(N)}(s)]ds.
\]
Note that Condition~\textbf{H1} implies that the initial moment $\bra{\mu^{(N)}_0, x^k}$ is uniformly bounded. 
Since $S_i^{(N)}(t) S_j^{(N)}(t) \le S_{i+j}^{(N)}(t)$, if follows that
\[
	F_k^{(N)}(t) \le C_k S_{k-1}^{(N)}(t), \quad \Ex[F_k^{(N)}(t)] \le C_k \Ex[S_{k-1}^{(N)}(t)],
\]
and hence, 
\[
	\quad \Ex[F_k^{(N)}(t)] \le C_k \Ex[S_{k-1}^{(N)}(t)] \le C_k D_{k-1},\quad t \in [0, T],
\]
for some constant $C_k$ not depending on $N$. Then the desired uniform boundedness follows immediately by induction. The proof is complete.
\end{proof}

\begin{proof}[Proof of Theorem~\rm\ref{thm:moment}]
Based on formula~\eqref{explicit-solution}, we prove this theorem by induction. The case $k = 0$ is trivial. Assume for now that for $l =0, 1, \dots, k-1$, the sequence $S_l^{(N)}$ converges in probability to a differentiable function $m_l$ (as random elements in $\X$). We need to show that the sequence $S_k^{(N)}$ converges in probability to $m_k$ which satisfies the ODE~\eqref{ODE-moment}.

By the induction hypothesis, it is clear that 
\[
	F_k^{(N)}(t) \to k(\alpha + k - 1) m_{k-1}(t) + c k \sum_{i = 0}^{k-1} m_i(t) m_{k-i-1}(t) =: F(t) \quad \text{in probability.}
\]
Together with Lemma~\ref{lem:MTG-part}, it follows that the function $\Phi_k^{(N)}$ (defined in~\eqref{integral-form}) converges in probability to $\int_0^t F(s) ds.$ Therefore, $S_k^{(N)}$ converges in probability to the limit $m_k$ given by
\begin{align*}
	m_k(t) &= a_k e^{-kt} + \int_0^t F(s) ds - k \left(\int_0^t  \int_0^s F(\tau)d\tau  e^{ks} ds \right) e^{-kt} \\
	&= a_k e^{-kt}  +\left( \int_0^t F(s) e^{ks} ds \right) e^{-kt}	.
\end{align*}
Here we have used integration by parts.
We then conclude that $m_k(t)$ is differentiable, and thus, it satisfies the ODE~\eqref{ODE-moment}. The proof is complete.
\end{proof}

Next, we study the ODE~\eqref{ODE-moment} in more details.
\begin{lemma}\label{lem:limit-moment}
Define a sequence $\{C_{k, 0}\}_{k \ge 0}$ as follows
\[
	\begin{cases}
		C_{0,0} = 1, \\
		C_{k, 0} = (\alpha + k - 1) C_{k - 1, 0} + c \sum_{i = 0}^{k - 1} C_{i, 0} C_{k - i - 1, 0}, \quad k \ge 1.
	\end{cases}
\]
Then for each $k$, the limiting $k$th moment process $m_k(t)$ has the form
\[
	m_k(t) = C_{k, 0} + \sum_{i = 1}^k C_{k, i} e^{-i t},
\] 
where $C_{k,i}$ are constants.
In particular, $\lim_{t \to \infty} m_k(t) = C_{k, 0}$. In addition, it holds that
\[
	\sup_{t \ge 0} m_k(t) \le \Lambda_k,
\]
where $\{\Lambda_k\}$ is the sequence in Condition \textbf{H2}.
\end{lemma}
\begin{proof}
Again, we prove this lemma by induction. In the proof, we will use some fundamental results on ODEs quoted in Lemma~\ref{lem:ODE1} and Lemma~\ref{lem:ODE2} in Appendix~\ref{appendix-ODE}.

\underline{The case $k = 1$.} The first moment process $m_1(t)$ satisfies the following ODE 
\[
\begin{cases}
	m_1'(t) = (\alpha + c) - m_1(t),\\
	m_1(0) = \bra{\mu_0, x} = a_1. 
\end{cases}
\]
Solving the equation gives the explicit formula 
\[
	m_1(t) = (\alpha + c)(1 - e^{-t}) + a_1 e^{-t} = C_{1, 0} + C_{1,1} e^{-t}.
\]
In particular,
\[
	m_1(t) \le (\alpha + c) \vee a_1 = \Lambda_1.
\]

\underline{The case $k\ge 2$.} By induction, the ODE for $m_k(t)$
can be written as 
\begin{align*}
	m_k'(t) = - k m_k(t) + k C_{k, 0} + \sum_{i = 1}^{k-1} D_{k, i} e^{-i t},
\end{align*}
where 
\[
	C_{k, 0} = (\alpha + k - 1) C_{k-1, 0} + c \sum_{i = 0}^{k - 1} C_{i, 0} C_{k - i - 1, 0},
\]
and $\{D_{k, i}\}_{1 \le i \le k - 1}$ are constants.
This implies an explicit formula for $m_k(t)$.
For the upper bound, since $m_i(t) m_j(t) \le m_{i + j}(t)$, it follows that
\[
	m_k'(t) \le - k m_k(t) + k (\alpha + k - 1) m_{k - 1}(t) + k^2 c m_{k - 1}(t),
\]
from which we deduce that
\[
	m_k(t) \le (\alpha + k - 1 + c k) \Lambda_{k-1} \vee a_k = \Lambda_k.
\]
The lemma is proved.
\end{proof}

As an example, we calculate the first five limiting moment processes in case $\alpha = 1, c = 1$ with trivial initial condition $\mu_0 = \delta_1$, that is, $a_k = 1, k = 1, 2, \dots,$
\begin{align*}
	&m_1(t) = (\alpha + c)(1 - e^{-t}) + a_1 e^{-t} = 2 - e^{-t}, \\
	&m_2(t) = 8 - 8 e^{-t} + e^{-2t},\\
	&m_3(t) = 44  - 66 e^{-t}+ 18 e^{-2 t} + 5 e^{-3 t} ,\\
	&m_4(t) = 96  - 592 e^{-t}  + 256 e^{-2 t} + 112 e^{-3 t}  - 71 e^{-4 t},\\
	&m_5(t) = 2312     -  5780 e^{-t} + 3460 e^{-2 t}+ 1880 e^{-3 t}- 2530 e^{-4 t} + 659 e^{-5 t}.
\end{align*}

We are now ready to state the main result of this paper.

\begin{theorem}\label{thm:main-H1-H2}
Assume that Conditions \textbf{H1} and \textbf{H2} are satisfied. Then for any $T>0$, the sequence of empirical measure processes $\mu_t^{(N)}$ converges in probability in $\cC([0,T], \cP(\R_{\ge 0}))$ to a continuous probability measure-valued process $\mu_t$ as $N \to \infty$. Here $\mu_t$ is the unique measure whose moments are given by $\{m_k(t)\}_{k = 1}^\infty$.
\end{theorem}
\begin{proof}
	Under Conditions \textbf{H1} and \textbf{H2}, Theorem~\ref{thm:moment} and Lemma~\ref{lem:limit-moment} imply that for each $t$, the sequence of limit moments $\{m_k(t)\}$ satisfies 
	\[
		\sum_{k = 1}^\infty m_k(t)^{-\frac{1}{2k}}  \ge \sum_{k = 1}^\infty \Lambda_k^{-\frac1{2k}} = \infty.
	\]
Therefore, there is a unique probability measure $\mu_t$ on $[0, \infty)$ whose moments are $\{m_k(t)\}$. The process $(\mu_t)_{t \ge 0}$ is continuous because $\mu_t$ is determined by moments. It follows from Theorem~\ref{thm:A} that the sequence $\mu_t^{(N)}$ converges in probability to $\mu_t$ in $\cC([0,T], \cP(\R_{\ge 0}))$, for each $T > 0$. The proof is complete.
\end{proof}

\subsection{Beta Laguerre ensembles at high temperature}\label{subsec:bLE}
For $\beta$LEs, in the regime where $\beta N \to 2c \in (0, \infty)$, the limiting behavior of the empirical distributions has been studied in \cite{Allez-Wishart-2013, Trinh-Trinh-2021}. It was shown that as $N \to \infty$, the empirical distribution 
\[
	L_N = \frac1N \sum_{i = 1}^N \delta_{\lambda_i}
\]
converges weakly to the probability measure $\nu_{\alpha, c}$ which is the probability measure of associated Laguerre orthogonal polynomials (model II) \cite{Trinh-Trinh-2021}. It is the spectral measure of the following Jacobi matrix (symmetric tridiagonal matrix)
\[
	J_{\alpha, c} = 
	\begin{pmatrix}
		\sqrt{\alpha + c} \\
		\sqrt{c + 1}	& \sqrt{\alpha + c + 1}\\
		&\ddots	&\ddots
	\end{pmatrix}
	\begin{pmatrix}
		\sqrt{\alpha + c } & \sqrt{c + 1}	\\
		&\sqrt{\alpha + c + 1}	&\sqrt{c+2}\\
		&&\ddots	&\ddots
	\end{pmatrix},
\]
that is, the measure $\nu_{\alpha, c}$ is determined by moments with moments given by 
\[
	\bra{\nu_{\alpha, c}, x^k} = (J_{\alpha, c})^k(1,1) =: u_k, \quad k =0,1,2,\dots.
\]
The density and the Stieltjes transform of $\nu_{\alpha, c}$ were calculated in \cite{Ismail-et-al-1988} 
\begin{align*}
	\nu_{\alpha, c}(x) &= \frac{1}{\Gamma(c+1) \Gamma(c+\alpha)} \frac{x^{\alpha-1} e^{-x}}{|\Psi(c, 1-\alpha; x e^{-i \pi})|^2}, \quad x \ge 0,\\
	S_{\nu_{\alpha, c}}(z) &= \int_0^\infty \frac{\nu_{\alpha, c} (x) dx}{x - z} = \frac{\Psi(c+1, 2-\alpha; -z)}{\Psi(c, 1-\alpha; -z)}, \quad z \in \C \setminus \R.
\end{align*}
Here $\Psi(a, b; z)$ is Tricomi's confluent hypergeometric function.

On the other hand, using ideas in \cite{Trinh-Shirai-2015}, we can show that the sequence of moments $\{u_k\}$ satisfies the self-convolutive equation 
\begin{equation}\label{self-convolutive}
	u_k = (\alpha + k - 1) u_{k - 1} + c \sum_{i = 0}^{k -1} u_i u_{k - i -1}, \quad k = 1,2,\dots.
\end{equation}
Therefore the sequence $\{C_{k,0}\}$ in Lemma~\ref{lem:limit-moment} coincides with the sequence of moments $\{u_k\}$ of $\nu_{\alpha, c}$. Note that from the self-convolutive equation, we can also calculate explicitly the density of $\nu_{\alpha, c}$  by using the result in \cite{Martin-Kearney-2010}. Thus, Theorem~\ref{thm:main-H1-H2} and Lemma~\ref{lem:limit-moment} imply that the limit process $\mu_t$ satisfies 
\[
	\lim_{t \to \infty} \mu_t = \nu_{\alpha, c},
\]
proving the convergence (ii) in the diagram~\eqref{diagram}.

\appendix

\section{Convergence of probability measure-valued processes}\label{appendix-moment}

Let $\Y = \cC([0, T], \cP(\R))$ be the space of continuous mappings $\mu \colon [0,T] \to \cP(\R)$ endowed with the topology of uniform convergence, where $\cP(\R)$ is the space of probability measures on $\R$ endowed with the topology of weak convergence. For definiteness, we consider the L\'evy--Prokhorov metric on $\cP(\R)$ which makes it a complete and separable metric space. Then $\Y$ can be metrizable to become a complete separable metric space. Recall that $\X = \cC([0, T], \R)$ is the space of continuous functions on $[0, T]$ endowed with the uniform norm. We are going to show the following result which can be roughly stated as the convergence of moments implies the convergence of measures at the process level.

\begin{theorem}\label{thm:A}

Let $\mu^{(N)}$ be a sequence of random elements on $\Y$. Assume that for each $k$, the $k$th moment process $\bra{\mu^{(N)}(t), x^k}$ is an $\X$-valued random element converging in probability to a non-random limit $m_k(t)$. For each $t \in [0,T]$, let $\mu(t)$ be a probability measure having moments $\{m_k(t)\}_{k \ge 1}$. Assume further that the measure $\mu(t), t \in [0,T],$ is determined by moments. Then $\mu = (\mu(t))_{0 \le t \le T}$ is an element in $\Y$, and the sequence $\mu^{(N)}$ converges in probability to $\mu$ as $N \to \infty$ as $\Y$-valued random elements.
\end{theorem}

Analogous to the case of random probability measures case (\cite[Lemma~2.2]{Trinh-ojm-2018}), the above theorem follows directly from the following deterministic result. 
\begin{lemma}
	Let $\{\mu^{(N)}(t)\}$ be a sequence in $\Y$ such that for each $k=1,2,\dots$, the sequence $\{\bra{\mu^{(N)}(t), x^k}\} \subset \X$ converges uniformly to a limit $m_k(t)$. Assume that for each $t \in [0,T]$, the sequence of moments $\{m_k(t)\}$ uniquely determines the probability measure $\mu(t)$. Then $\mu=(\mu(t))_{0 \le t \le T} \in \Y$ and the sequence $\{\mu^{(N)}\}$ converges to $\mu$.
\end{lemma}
\begin{proof}
Since the functions $m_k(t)$ are continuous and for each $t$ the measure $\mu(t)$ is determined by moments, it is clear that $\mu$ is an element of $\Y$.
We will show that $\{\mu^{(N)}\}$ converges to $\mu$ by contradiction. Indeed, assume for contradiction that the sequence $\{\mu^{(N)}\}$ does not converge to $\mu$. Then we can find a subsequence $\{N_l\} \subset \N$, a sequence $\{t_l\} \subset [0,T]$ converging to $t$ such that 
$
	\mu^{N_l} (t_l)
$
does not converge to $\mu (t)$. However, each moment of $\mu^{N_l}({t_l}) $ converges to that of $\mu(t)$ by the uniform convergence assumption, implying the weak convergence of probability measure, which is a contradiction. The lemma is proved.
\end{proof}

\section{Fundamental results on ODEs}\label{appendix-ODE}

\begin{lemma}\label{lem:ODE1}
The solution to the initial value ODE
\[
	\phi'(t) = -k \phi(t) + F(t), \quad \phi(0) = \phi_0,
\]
is of the form
\[
	\phi(t) =\left (\phi_0 + \int_0^t F(s) e^{k s} ds \right) e^{-kt}.
\]
Consequently, if $F(t) \le G(t), t \ge 0$, then 
\[
	\phi(t) \le \psi(t), \quad (t\ge 0),
\]
where $\psi(t)$ is the solution to the equation 
\[
	\psi'(t) = -k \psi(t) + G(t), \quad \psi(t) = \phi_0.
\]
\end{lemma}

\begin{lemma}\label{lem:ODE2}
	The solution to the initial value ODE 
	\[
		\phi'(t) = k( -\phi + D),\quad \phi(0) = C,
	\]
where $k, C, D > 0$ are constants, is given by 
\[
	\phi(t) = D(1 - e^{-kt}) + C e^{-kt}.
\]
Consequently, 
\[
	\sup_{t \ge 0} \phi(t) \le C \vee D.
\]
\end{lemma}

\bigskip
\noindent{\textbf{Acknowledgements.}} 
The authors would like to thank the reviewer(s) for their thorough review and helpful suggestions.
The authors would like also to thank Prof. Fumihiko Nakano and Dr.~Sergio Andraus for useful discussions.  This work is supported by JSPS KAKENHI Grant Number JP19K14547 (K.D.T.).

\begin{footnotesize}


\begin{thebibliography}{10}
\providecommand{\url}[1]{{#1}}
\providecommand{\urlprefix}{URL }
\expandafter\ifx\csname urlstyle\endcsname\relax
  \providecommand{\doi}[1]{DOI~\discretionary{}{}{}#1}\else
  \providecommand{\doi}{DOI~\discretionary{}{}{}\begingroup
  \urlstyle{rm}\Url}\fi

\bibitem{Allez-Wishart-2013}
Allez, R., Bouchaud, J.P., Majumdar, S.N., Vivo, P.: Invariant
  {$\beta$}-{W}ishart ensembles, crossover densities and asymptotic corrections
  to the {M}ar\v cenko-{P}astur law.
\newblock J. Phys. A \textbf{46}(1), 015,001, 22 (2013).
\newblock \urlprefix\url{https://doi.org/10.1088/1751-8113/46/1/015001}

\bibitem{Baker-Forrester-1997}
Baker, T.H., Forrester, P.J.: The {C}alogero-{S}utherland model and generalized
  classical polynomials.
\newblock Comm. Math. Phys. \textbf{188}(1), 175--216 (1997).
\newblock \doi{10.1007/s002200050161}.
\newblock \urlprefix\url{https://doi.org/10.1007/s002200050161}

\bibitem{Billingsley}
Billingsley, P.: Probability and measure, third edn.
\newblock Wiley Series in Probability and Mathematical Statistics. John Wiley
  \& Sons, Inc., New York (1995).
\newblock A Wiley-Interscience Publication

\bibitem{Bru-1989}
Bru, M.F.: Diffusions of perturbed principal component analysis.
\newblock J. Multivariate Anal. \textbf{29}(1), 127--136 (1989).
\newblock \doi{10.1016/0047-259X(89)90080-8}.
\newblock \urlprefix\url{https://doi.org/10.1016/0047-259X(89)90080-8}

\bibitem{Bru-1991}
Bru, M.F.: Wishart processes.
\newblock J. Theoret. Probab. \textbf{4}(4), 725--751 (1991).
\newblock \doi{10.1007/BF01259552}.
\newblock \urlprefix\url{https://doi.org/10.1007/BF01259552}

\bibitem{Duvillard-Guionnet-2001}
Cabanal~Duvillard, T., Guionnet, A.: Large deviations upper bounds for the laws
  of matrix-valued processes and non-communicative entropies.
\newblock Ann. Probab. \textbf{29}(3), 1205--1261 (2001).
\newblock \doi{10.1214/aop/1015345602}.
\newblock \urlprefix\url{https://doi.org/10.1214/aop/1015345602}

\bibitem{Cepa-Lepingle-1997}
C\'{e}pa, E., L\'{e}pingle, D.: Diffusing particles with electrostatic
  repulsion.
\newblock Probab. Theory Related Fields \textbf{107}(4), 429--449 (1997).
\newblock \doi{10.1007/s004400050092}.
\newblock \urlprefix\url{https://doi.org/10.1007/s004400050092}

\bibitem{Demni-2007}
Demni, N.: The {L}aguerre process and generalized {H}artman-{W}atson law.
\newblock Bernoulli \textbf{13}(2), 556--580 (2007).
\newblock \doi{10.3150/07-BEJ6048}.
\newblock \urlprefix\url{https://doi.org/10.3150/07-BEJ6048}

\bibitem{Demni-2007-arxiv}
Demni, N.: Radial dunkl processes: Existence and uniqueness, hitting time, beta
  processes and random matrices.
\newblock arXiv preprint arXiv:0707.0367  (2007)

\bibitem{DE02}
Dumitriu, I., Edelman, A.: Matrix models for beta ensembles.
\newblock J. Math. Phys. \textbf{43}(11), 5830--5847 (2002).
\newblock \urlprefix\url{http://dx.doi.org/10.1063/1.1507823}

\bibitem{DE06}
Dumitriu, I., Edelman, A.: Global spectrum fluctuations for the
  {$\beta$}-{H}ermite and {$\beta$}-{L}aguerre ensembles via matrix models.
\newblock J. Math. Phys. \textbf{47}(6), 063,302, 36 (2006).
\newblock \urlprefix\url{http://dx.doi.org/10.1063/1.2200144}

\bibitem{Dunkl-1989}
Dunkl, C.F.: Differential-difference operators associated to reflection groups.
\newblock Trans. Amer. Math. Soc. \textbf{311}(1), 167--183 (1989).
\newblock \doi{10.2307/2001022}.
\newblock \urlprefix\url{https://doi.org/10.2307/2001022}

\bibitem{Dunkl-1991}
Dunkl, C.F.: Integral kernels with reflection group invariance.
\newblock Canad. J. Math. \textbf{43}(6), 1213--1227 (1991).
\newblock \doi{10.4153/CJM-1991-069-8}.
\newblock \urlprefix\url{https://doi.org/10.4153/CJM-1991-069-8}

\bibitem{Durrett-book}
Durrett, R.: Probability---theory and examples, \emph{Cambridge Series in
  Statistical and Probabilistic Mathematics}, vol.~49.
\newblock Cambridge University Press, Cambridge (2019).
\newblock \doi{10.1017/9781108591034}.
\newblock \urlprefix\url{https://doi.org/10.1017/9781108591034}.
\newblock Fifth edition of [ MR1068527]

\bibitem{Trinh-ojm-2018}
Duy, T.K.: On spectral measures of random {J}acobi matrices.
\newblock Osaka J. Math. \textbf{55}(4), 595--617 (2018).
\newblock \urlprefix\url{https://projecteuclid.org/euclid.ojm/1539158661}

\bibitem{Trinh-Shirai-2015}
Duy, T.K., Shirai, T.: The mean spectral measures of random {J}acobi matrices
  related to {G}aussian beta ensembles.
\newblock Electron. Commun. Probab. \textbf{20}, no. 68, 13 (2015).
\newblock \urlprefix\url{http://dx.doi.org/10.1214/ECP.v20-4252}

\bibitem{Graczyk-Jacek-2014}
Graczyk, P., Ma\l~ecki, J.: Strong solutions of non-colliding particle systems.
\newblock Electron. J. Probab. \textbf{19}, no. 119, 21 (2014).
\newblock \doi{10.1214/EJP.v19-3842}.
\newblock \urlprefix\url{https://doi.org/10.1214/EJP.v19-3842}

\bibitem{Ismail-et-al-1988}
Ismail, M.E.H., Letessier, J., Valent, G.: Linear birth and death models and
  associated {L}aguerre and {M}eixner polynomials.
\newblock J. Approx. Theory \textbf{55}(3), 337--348 (1988).
\newblock \urlprefix\url{https://doi.org/10.1016/0021-9045(88)90100-1}

\bibitem{Katori-Tanemura-2004}
Katori, M., Tanemura, H.: Symmetry of matrix-valued stochastic processes and
  noncolliding diffusion particle systems.
\newblock J. Math. Phys. \textbf{45}(8), 3058--3085 (2004).
\newblock \doi{10.1063/1.1765215}.
\newblock \urlprefix\url{https://doi.org/10.1063/1.1765215}

\bibitem{Konig-OConnell-2001}
K\"{o}nig, W., O'Connell, N.: Eigenvalues of the {L}aguerre process as
  non-colliding squared {B}essel processes.
\newblock Electron. Comm. Probab. \textbf{6}, 107--114 (2001).
\newblock \doi{10.1214/ECP.v6-1040}.
\newblock \urlprefix\url{https://doi.org/10.1214/ECP.v6-1040}

\bibitem{Martin-Kearney-2010}
Martin, R.J., Kearney, M.J.: An exactly solvable self-convolutive recurrence.
\newblock Aequationes Math. \textbf{80}(3), 291--318 (2010).
\newblock \doi{10.1007/s00010-010-0051-0}.
\newblock \urlprefix\url{https://doi.org/10.1007/s00010-010-0051-0}

\bibitem{Pastur-book}
Pastur, L., Shcherbina, M.: Eigenvalue distribution of large random matrices,
  \emph{Mathematical Surveys and Monographs}, vol. 171.
\newblock American Mathematical Society, Providence, RI (2011).
\newblock \doi{10.1090/surv/171}.
\newblock \urlprefix\url{https://doi.org/10.1090/surv/171}

\bibitem{Rogers-Shi-1993}
Rogers, L.C.G., Shi, Z.: Interacting {B}rownian particles and the {W}igner law.
\newblock Probab. Theory Related Fields \textbf{95}(4), 555--570 (1993).
\newblock \doi{10.1007/BF01196734}.
\newblock \urlprefix\url{https://doi.org/10.1007/BF01196734}

\bibitem{Rosler-1999}
R\"{o}sler, M.: Positivity of {D}unkl's intertwining operator.
\newblock Duke Math. J. \textbf{98}(3), 445--463 (1999).
\newblock \doi{10.1215/S0012-7094-99-09813-7}.
\newblock \urlprefix\url{https://doi.org/10.1215/S0012-7094-99-09813-7}

\bibitem{Rosler-Voit-1998}
R\"{o}sler, M., Voit, M.: Markov processes related with {D}unkl operators.
\newblock Adv. in Appl. Math. \textbf{21}(4), 575--643 (1998).
\newblock \doi{10.1006/aama.1998.0609}.
\newblock \urlprefix\url{https://doi.org/10.1006/aama.1998.0609}

\bibitem{Trinh-Trinh-2021}
Trinh, H.D., Trinh, K.D.: Beta {L}aguerre ensembles in global regime.
\newblock Osaka J. Math.  (2021 (to appear), arXiv preprint arXiv:1907.12267)

\end{thebibliography}
\end{footnotesize}
\end{document}